\documentclass[article,noinfoline]{imsart}
\usepackage{graphicx}
\usepackage{jr,comment}

\usepackage{color}
\newtheorem{Assumption}{Assumption}
\startlocaldefs
\def\bb#1{\mathbb{#1}}

\def\en{\infty}
\def\summ#1#2#3{\sum_{#1=#2}^{#3}}

\newcommand{\bdf}{\boldsymbol}
\endlocaldefs

\begin{document}

\begin{frontmatter}

\title{ Empirical Bayes improvement of Kalman filter type of estimators}

\author{\fnms{Eitan} \snm{Greenshtein}\ead[label=e2]{eitan.greenshtein@gmail.com}}
\address{Israel Census Bureau of Statistics; \printead{e2}}
\affiliation{Israel Census Bureau of Statistics}

\author{\fnms{Ariel} \snm{Mansura,}\ead[label=e1]{ariel.mansura@boi.org.il}}
\address{Bank of Israel; \printead{e1}}
\affiliation{Bank of Israel}

\author{\fnms{Ya'acov} \snm{Ritov}
\ead[label=e3]{yaacov.ritov@gmail.com}}
\address{The Hebrew University of Jerusalem; \printead{e3}}
\affiliation{The Hebrew University of Jerusalem}

\runauthor{ Mansura, Greenshtein, Ritov}

\maketitle

\begin{abstract}
We consider the problem of estimating the means $\mu_i$ of
$n$ random variables
$Y_i \sim N(\mu_i,1)$, $i=1,\ldots ,n$. Assuming some structure on the $\mu$ process, e.g., a state space model,
one may use a summary statistics for the contribution of the rest of the observations  to the  estimation of $\mu_i$. The most important example for this is the Kalman filter.
We introduce  a non-linear improvement of  the standard weighted average of the given summary statistics and $Y_i$ itself, using empirical  Bayes methods.
The improvement is obtained under mild assumptions. It is strict  when the process that governs the states $\mu_1,\ldots,\mu_n $  is not a  linear  Gaussian state-space  model.
We consider both the sequential and the retrospective estimation problems.
\end{abstract}

\end{frontmatter}

\section{Introduction and Preliminaries \label{sec:int} }

We consider the estimation under squared error loss of a vector $\mu_1,\dots,\mu_n$ observed with Gaussian additive error:  $Y_i=\mu_i+\eps_i$, $i=1,\dots,n$, where  $\eps_1,\dots,\eps_n$ are \iid\  $N(0,1)$. It is natural in our applications to consider the index $i$ as denoting time, and regard $\mu_1,\dots,\mu_n$ as a realization of a stochastic process. We  analyze, accordingly, two main setups. In the first, the estimation is done retrospectively, after all of $Y_1,\dots,Y_n$ are observed. The second case  is of  sequential estimation, where $\mu_i$ should be estimated at time $i$, after observing $Y_1,\dots,Y_i$.  Let $\scd_i$ be the data set  based on which $\mu_i$ is estimated, excluding the $i$th observation itself. That is, $\scd_i = \{1,\dots,i-1\}$ in the sequential  case, and $ \scd_i =\{j: 1\le j\le n,j\ne i\}$ when the estimation is retrospective.

We could consider a more general situation in which the observations are $(Y_i,\bdf{X}_i)$, $i=1,\dots,n$, where the $\bdf{X}_i$s are observed covariates and
\begin{equation*}\hat{\mu}_i = \sum_{j \in \scd_i\union\{i\}} \Bigl(\beta_{ij}  Y_{j} + \bdf{\beta}^x_{ij} \bdf{X}_{j}\Bigr).
\end{equation*}
However, to simplify the presentation, we  discuss only the situation without observed covariates:
\begin{equation}\label{muhat1}
\hat{\mu}_i = \sum_{j \in \scd_i\union\{i\}} \beta_{ij}  Y_{j}.
\end{equation}

When $\mu_1,\mu_2,\dots$ are a realization of a   Gaussian process, the optimal estimator for $\mu_i$ based on the data set $\scd_i\union\{i\}$
is indeed linear, and is given by the Kalman filter (KF).  However, in more
general state space models, and certainly when the model is misspecified,  the Kalman filter, or any other linear scheme, are not optimal. Yet, they may be taken as a reasonable starting point for the construction of a better estimator.  We consider in this paper   an empirical Bayes improvement of a  given linear filter which is nonparametric and does not depend on  structural  assumptions.

The  linear estimator $\hat\mu_i$ in \eqref{muhat1} can be be considered as a weighted average of two components, $Y_i$, and an estimator $\ti\mu_i$ based on all the observations available at time $i$ excluding the $i$th observations itself:
$$\tilde{\mu}_i = \sum_{j \in {\scd}_i } \tilde{\beta}_{ij}  Y_{j} .$$
In the Gaussian case, $\ti\mu_i$ and $\hat\mu_i$ are typically the sufficient statistics for $\mu_i$ given the data in $\scd_i$ and $\scd_i\union\{i\}$ respectively.
In the sequential Gaussian case the estimator $\tilde{\mu}_i$
is called  the optimal  one step ahead predictor of $\mu_i$
while $\hat{\mu}_i$ is the KF estimator of $\mu_i$,
$i=1,\ldots,n$. For background about the KF, state-space
models, and general facts about time series see, e.g., Brockwell and Davis (1991). We will hardly use that theory in the following development, since we aim for results that are true regardless
on whether various common state-space assumptions hold. In the sequel, when we want to emphasize that $\tilde{\mu}_i$ and $\hat{\mu}_i$ are the standard KF estimators
we will write $\tilde{\mu}^K_i$ and $\hat{\mu}^K_i$, but the following derivation is for a general pair $\tilde{\mu}_i$ and $\hat{\mu}_i$.

Our goal in this paper is to use $\ti\mu_i$ as a basis for the construction of an estimator which   improves upon $\hat\mu_i$. In fact, we try to find the best estimator of the form:
\begin{align}
\label{eqn:imp0}\hat{\mu}_{ig}& = \tilde{\mu}_i + g(Y_i- \tilde{\mu}_i),\quad i=1,\dots,n.\\\intertext{ Let}
\delta &\equiv \argmin_g E \sum_{i=1}^n (\hat\mu_{ig}-\mu_i)^2
\label{eqn:imp}
\end{align}

Thus, we use a simple coordinate-wise function, as was introduced by Robbins (1951)  in the context of compound decision:

\begin{Definition}\label{def:1} A function ${\bdf f}: \R^n \rightarrow \R^n$   is called simple coordinate-wise function, if it has the representation
$\bdf{f}(X_1,\ldots ,X_n)=\bigl(f(X_1),\ldots ,f(X_n)\bigr)$ for  some $f: \R \rightarrow \R$.
\end{Definition}

Our improvement, denoted $\bdf{\delta}(\cdot)$, is a simple coordinate-wise function of $(\bdf{Y}-\bdf{\tilde{\mu}})$.
In the theory of compound decision and empirical Bayes, the search for an optimal simple-coordinate-wise  decision
function is central. We elaborate in the next section.
The improved  estimator $\mu_{i \delta}$ is denoted $\mu_i^I$, and in vector notations we write in short
 $$\bdf{\mu}^I = \bdf{\tilde{\mu}} + \bdf{\delta}.$$

\subsection{Empirical Bayes and non-exchangeable observations} The ideas of empirical Bayes (EB) and compound decision (CD) procedures were developed by Robbins (1951, 1955, 1964), see  the review papers of Copas (1969)
and Zhang (2003), and the paper of Greenshtein and Ritov (2008) for results relating
compound decision, simple coordinate-wise  decision functions and permutational invariant
decision functions.

The classical EB/CD theory is restricted to independent exchangeable observations and to  permutation invariant procedures, and in particular it excludes the utilization of explanatory variables. Fay and Herriot (1979) suggested a way to extend the ideas of parametric EB (i.e, linear decision functions corresponding to Gaussian measurement and prior) to
handle covariates.
Recently, there is an effort to extend the EB ideas, so they may be
incorporated in the presence of covariates also in the context of non-parametric
EB, see, Jiang and Zhang (2010),
Cohen et al. (2013), and Koenker and Mizera (2013).
Our paper may be viewed as a continuation of this effort.

The above papers extended the discussion to the situation where the observations, due to the covariates, are not exchangeable. However, the estimated parameters themselves, $\mu_1,\dots,\mu_n$, are permutation invariant.  Thus, in all these problems,  centering each response by a linear transformation of the covariates transforms the problem into a classical EB problem. In our setup of a time series, the estimated variables are not permutation invariant, and   the explanatory variables of $Y_i$ are the available observations $Y_j, \; j \neq i$,
so there is an obvious strong dependence between the response variables and the covariates and the response is degenerate conditional on the covariates.

 Furthermore,
 in all the above mentioned papers the extension of EB ideas to handle covariates
 is done in a retrospective setup, where all the observations  are given in advance.
 Under the time series structure that we study, it is natural to consider real time sequential   estimation of the $\mu$'s.  In Section \ref{sec:seq} we consider the sequential case, where at stage $i$ the decision function should be approximated
based on the currently available data. Our analysis would be based on an extension of  Samuel (1965).
The
retrospective case  is simpler and will
be treated first in Section \ref{sec:nonseq}.  A small simulation study is presented in Section \ref{sec:sim}, and a real data example is discussed in Section \ref{sec:real}.

\subsection{Estimated simple coordinate-wise function}

Most EB/CD solutions involve simple coordinate-wise functions. By the nature of the problem, these functions are  estimated from the data, which is used symmetrically.  \begin{equation}
\label{scwf}\bdf{\hat f}(X_1,\dots,X_n)=
\bigl(\hat f(X_1;X_{(1)},\dots,X_{(n)}),\dots,\hat f(X_n;X_{(1)},\dots,X_{(n)})\bigr),
\end{equation}
where  $X_{(1)}\leq\dots\leq X_{(n)}$ are the ordered statistics

Unfortunately, any permutation invariant function can be written in this way. Suppose for simplicity that $X_1,\dots,X_n$ are real. Let $\bdf\psi(X_1,\dots,X_n):\R^n\to\R^n$ be a permutation invariant function. Let $\ind(\cdot)$ be the indicator function. It is possible to write $\bdf\psi=\bdf{\hat f}$ as in \eqref{scwf}, with
\begin{equation*}
\hat f(x; X_{(1)},\dots,X_{(n)}) = \sum \bdf\psi_{i}\bigl(X_{(1)},\dots,X_{(n)}\bigr)\ind(x=X_{(i)}),
\end{equation*}
or a smooth version of this function.

Actually, any function that is estimated from the data and is used only on that data can be written as a simple coordinate-wise  function.

Intuitively, the set of simple coordinate-wise functions is a strict subset of the set of permutation invariant functions. We therefore consider a function $\bdf{\hat f}$ as simple coordinate-wise function if it approximates a function $\bdf f$ that is simple coordinate-wise function by Definition \ref{def:1}.
This later function may be random (i.e., a stochastic process), with non-degenerate asymptotic distribution.

\subsection{Assumptions}
The performance of our estimators will be measured by their mean squared error loss, in vector notation:
$E ||\bdf{\mu}^I-\bdf{\mu}||^2$, $ E ||\tilde{\bdf{\mu}}-\bdf{\mu}||^2$, and $E ||\hat{\bdf{\mu}}-\bdf{\mu}||^2$. Let $\scf_i $ be the smallest $\sig$-field under which $Y_j$, $j\in \scd_i$ are measurable. The dependency of different objects on $n$ will be suppressed, when there will be no danger of confusion.

\begin{Assumption}\label{ass1}
For every $i=1,\dots,n$, the estimator $\ti\mu_i$ is $\scf_i$ measurable. It is  Lipschitz in $Y_j$ with a constant $\rho_{|i-j|}$, where $\limsup_{M\to\en}M^2\rho_M<1$. That is: For $j\in \scd_i$ let $\ti\mu_i'$ be $\ti\mu_i$, but computed with $Y_j$  replaced by $Y_j+d$. Then, $|\ti\mu_i'-\ti\mu_i|\le \rho_{|i-j|}d$.
\end{Assumption}
This condition is natural when $\ti\mu$ is  KF for a stationary Gaussian process, where typically $\beta_{ij}$ decreases exponentially with $|i-j|$. The main need for generalizing the KF is to include filters which are based on estimated parameters.

The Kalman filter for an ergodic process also satisfies the following condition. It has no real importance for our results, except giving a standard benchmark.

\begin{Assumption}\label{ass2}
Suppose that  there is a  $\al_n\in\scf_n$, $\al_n<1$:
\begin{equation}\label{asymAl}
\hat{\mu}_i= \alpha_n \tilde{\mu}_i + (1-\alpha_n)Y_i + \zeta_i, \;\text{where } E \zeta_i^2\to 0,
\end{equation}
as $n\rightarrow \infty$, $0<\liminf i/n\leq\limsup i/n<1 $.
\end{Assumption}
\noindent{\bf Remark:}
Our major example is the ergodic normal state-space model. If the assumed model is correct, and $\ti\mu_i$ and $\hat\mu_i$ are the optimal estimators, then $\ti\mu_i$ is a sufficient statistics for $\mu_i$ given $\scd_i$. The  estimators satisfy \eqref{asymAl}
with $\alpha_n\equiv (1+\tau^2)^{-1}$, where $\tau^2$ is the asymptotic variance of $\mu_i$ given $\tilde{\mu}_i$. In the iterative Kalman filter method for computing $\hat{\mu}_i$, with some abuse of notation, the  values
$\alpha_i=(1+\tau^2_i)^{-1}$ are computed, with $\tau^2_i$ the variance of $\mu_i$ given $\tilde{\mu}_i$, and we have $\hat{\mu}_i= \alpha_i \tilde{\mu}_i + (1-\alpha_i)Y_i$.

By considering the functions
$g(z)\equiv 0$ and $g(z)=(1-\alpha)z$ in \eqref{eqn:imp0}, it is easy to see that  $\bdf{\mu}^I$ has asymptotically   mean squared error not larger than $\tilde{\bdf{\mu}}$  and
$\hat{\bdf{\mu}}$, respectively. In fact, we argue that unless the process is asymptotically Gaussian, there is a strict improvement.

The derivation of the Kalman filter  is based on an assumed stochastic model for the sequence $\mu_1,\dots,\mu_n$. Very few properties of the the process are relevant, and it is  irrelevant to our discussion whether the model is true or not. However, we do need some tightness. We expect that typically $|\mu_i-\mu_{_{i-1}}|$ is not larger than $\log n$, and $\ti\mu_i$ is sensible at least as $\ti\mu_i\equiv Y_{i-1}$. Since $\max|Y_i-\mu_i|=\o_p(\sqrt{\log n})$, the next condition is natural:

\begin{Assumption}\label{ass3}
It holds:
\eqsplit{
    \frac1n\summ i1n P\bigl( |Y_i-\ti\mu_i|> \log n\bigr)\le \frac1{(\log n)^{8}}.
    }
\end{Assumption}

\section{ Retrospective estimation \label{sec:nonseq}}

Denote,
 \begin{equation}\label{nuZ}
  \begin{split}
    Z_i&= Y_i- \tilde{\mu}_i;  \\
    \nu_i&= \mu_i -\tilde{\mu}_i,\qquad i=1,\dots,n.
   \end{split} \end{equation}
Clearly, $Z_i=\nu_i+\eps_i$.  Since $\eps_i$ is independent both of $\mu_1,\dots,\mu_n$ and of $\eps_j$, $j\ne i$, it  is independent of $\nu_i$. Thus,  the conditional distribution of $Z_i$ given $\nu_i$ is $N(\nu_i,1)$.
However, this is not a regular EB problem. It is not so even for the regular KF.  Write $\ti{\bdf\mu}=B\bdf{Y}=B\bdf{\mu}+B\bdf{\eps}$. Then $\bdf{\nu}=(I-B)\bdf{\mu}-B\bdf{\eps}$. It is true that  $\bdf{Z}=\bdf{\nu}+\bdf{\eps}$, but the vectors $\bdf{\nu}$ and $\bdf{\eps}$ are not independent. Hence $\bdf{Z}|\bdf{\nu}\not\dist N_n(\bdf{\nu},I_n)$. Yet, we rely only on the marginal distributions of $Z_i|\nu_i$, $i=1,\dots,n$.

To elaborate,
 \eqsplit{
    \begin{pmatrix}\bdf Y \\ \bdf\mu\end{pmatrix}
    =\begin{pmatrix}(I-B)^{-1} & 0\\B(I-B)^{-1} & I\end{pmatrix}
    \begin{pmatrix}\bdf Z\\\bdf\nu \end{pmatrix}.
  }
Therefore, the joint density of $Z$ and $\nu$ is proportional to
 $$
    f_{\bdf \mu}\bigl(\bdf\nu+B(I-B)^{-1}\bdf Z\bigr)\exp\bigl(-\|\bdf Z-\bdf \nu\|^2/2\bigr),
$$
where $f_{\bdf\mu}$ is the joint density of the vector $\bdf\mu$. Clearly, unless $f_{\bdf\mu}$ is multivariate normal, the conditional density of $\bdf Z$ given $\bdf\nu$ is not multivariate standard normal.

\begin{example}
Suppose $n=2$,  we observe $Y_0,Y_1$, and use $\ti\mu_i=\gamma Y_{1-i}$, $i=0,1$. Then
 \eqsplit{
    Z_i = Y_i-\gamma Y_{1-i}\quad&\implies Y_i = \frac{1}{1-\gamma^2}(Z_i+ \gamma Z_{1-i})
    \\
    \nu_i=\mu_i-\gamma Y_{1-i} \quad&\implies Y_i=\frac1{\gamma}(\mu_{1-i}-\nu_{1-i})
    \\
    &\implies Z_i = \frac1\gamma (\mu_{1-i}-\nu_{1-i}-\gamma \mu_i+\gamma \nu_i).
  }
Suppose further that $\mu_i$ is finitely supported. It follows from the above calculations that the distribution of the vector $\bdf Z$ given the vector $\bdf \nu$ is finitely supported as well.
\end{example}

The estimator in vector notation is
$\bdf{\mu}^I= \bdf{\tilde{\mu}}+\bdf{\delta}$,
where $\bdf{\delta}=\bigl(\delta(Z_1),\ldots ,\delta(Z_n)\bigr)\t$. As discussed in the introduction,  simple coordinate-wise functions like
$\bdf{\delta}$ are central in EB and CD models. However, our decision function $\bdf{\mu}^I$ is not a simple coordinate-wise function of the observations. It is  a hybrid of non-coordinate-wise  function $\tilde{\bdf{\mu}}$ and a simple coordinate-wise  one, $\bdf{\delta}$. The  $\tilde{\bdf{\mu}}$ component accounts for the
non-coordinate-wise  information from the covariates, while $\bdf{\delta}$ aims to improve it in a coordinate-wise  way after the information from all other observations was accounted for by $\tilde{\bdf\mu}$.

By Assumption \ref{ass1}, the dependency between the $Z_i$s conditioned on $\bdf\nu$ is only local, and hence, if we consider a permutation invariant procedure, which treats neighboring observations and far away ones the same, the dependency disappears asymptotically, and we may consider only the marginal normality of the $\bdf\nu$. The basic ideas of  EB are helpful and
we get the  representation \eqref{eqn:brown} of ${\delta}$ as given below.

Let
\begin{equation*}
f_Z(z)= \frac1n \summ i1n \varphi(z-\nu_i),
\end{equation*}
where $\varphi$ is the standard normal density. Note that this is not a kernel estimator---the kernel is with fixed bandwidth and  $\nu_1,\dots,\nu_n$ are unobserved.
Let $I$ be uniformly distributed over $1,\ldots ,n$. Denote by $F^n$ the distribution of the random pairs $(\nu_I, \nu_I+\eta_I)$, where $\eta_1,\dots,\eta_n$ are \iid standard normal independent of the other random variables mentioned so far and the randomness is induced by the random index $I$ and the $\eta$s.  One marginal distribution of $F^n$ is the empirical distribution of $\nu_1,\dots,\nu_n$, while the density of the other is given by $f_Z$. We denote the marginals by $F_\nu^n$ and $F_Z^n$.  Finally, note that $Z_I$ given $\nu_I$ has the distributing of $F^n_{Z|\nu}$, i.e., the conditional distribution of $F^n$.

It is well known that asymptotically, the Bayes procedure for estimating $\nu_i$ given $Z_i$ is approximated by the Bayes procedure with  $F^n_\nu$ as prior, and it is determined by $f_Z$. The optimal simple coordinate-wise function $\delta=\delta^n$  depends only on marginal joint distribution of $(\nu_I,Z_I)$. In fact, it depends only on  $f_Z$.
As in Brown (1971) we have:
\begin{equation}
\label{eqn:brown}
\delta^n(z)=E_{F^n}(\nu_I|\nu_I+\eta_I=z)= z+ \frac{f_{Z}'(z)}{f_{Z}(z)},
\end{equation}
where $f'_{Z}$ is the derivative of $f_Z$.
The dependency on $n$ is suppressed in the notations.

Note that $\delta^n$ is a random function, and in fact, if $\mu_1,\dots,\mu_n$ is not an ergodic process, it may not have an asymptotic deterministic limit. Yet, it would be the object we estimate in \eqref{eqn:brown1} below.

It is of a special interest to characterize when $\delta=\delta^n$ is asymptotically linear, in which case  the improved estimator
 $\mu_i^I=\tilde{\mu}_i + \delta^n(Z_i)$  is asymptotically a  linear combination of $\tilde{\mu}_i$ and $Y_i$. Only in such a case the difference between
the loss of the improved estimator $\bdf{\mu}^I=\ti{\bdf\mu}+\bdf\delta$ and that of the  estimator $\hat{\bdf{\mu}}$ may be asymptotically of   $o(n)$.
It follows from \eqref{eqn:brown} that, the optimal decision $\del(Z)$ is  approximately $(1-\al)Z$, if and only if, $f_Z'/f_Z=(\log f_Z)'$ is approximately proportional to $z$. This happens only if $f_Z$ converges to a  Gaussian distribution. Since $f_Z$ is a convolution of a Gaussian kernel with the the prior, this can happen only if the prior is asymptotically Gaussian. In our setup where $F_\nu^n$ plays the role of a prior, in order to have asymptotically  linear improved estimator
we need that  $F^n_\nu$ converges weakly to  a normal distribution $G$.

The above is formally stated in the following Proposition \ref{prop:lin}.
\begin{proposition}
\label{prop:lin}
Under assumptions 1-3,
$n^{-1}E\|\bdf\mu^I-\bdf{\hat{\mu}}\|^2\to 0$  implies that  the sequence $F_\nu^n-N(0,(1-\al_n)/\al_n)$  converges weakly to the zero measure.
\end{proposition}

Given the observations $Z_1,\dots,Z_n$, let $\hat{F}^n_Z$ be the empirical distribution of $Z_1,\ldots ,Z_n$.
We will show that as $n \rightarrow \infty$ the `distance' between $\hat{F}^n_Z$ and $F^n_Z$ gets smaller, so that
$f_Z$ and and its derivative  may be replaced in \eqref{eqn:brown}  by appropriate kernel estimates based on
$\hat{F}_Z^n$, and yield
a good enough estimator  $\hat{\delta}^n$ of $\delta^n$. In the sequel we will occasionally drop the superscript $n$.
Note, that we do not assume that $F^n$ approaches some limit $F$ as $n \rightarrow \infty$, although this
is the situation if we assume that $Z_1,Z_2,\ldots $ is an ergodic stationary process, however our assumptions on that
process are milder.

We now state  the above formally.
Consider the two kernel estimators $\hat{f}_Z(z)=n^{-1}\sum_j K_\sig(Z_j-z)$ and $\hat{f}_Z'(z)=n^{-1}\sum_j K'_\sig(Z_j-z)$, where $K_\sig(z)={\sig^{-1}}K\bigl(z/{\sig}\bigr)$,  $\sig=\sig_n$. For simplicity, we use the same bandwidth to estimate both the density and its derivative.
We define the following estimator $\hat{\delta} \equiv \hat{\delta}_\sigma$ for $\delta$:

\begin{equation}
\label{eqn:brown1}
             \hat{\delta}(z)=\hat{\delta}^n_\sig(z) \equiv         z+ \frac{\hat{f}_Z'(z)}{\hat{f}_Z(z)}.
\end{equation}
Brown and Greenstein (2009) used the normal kernel, we prefer to use the logistic kernel $2(e^{x}+e^{-x})^{-1}$ (it is the derivative of the logistic cdf, $(1+e^{-2x})^{-1}$, and hence its  integral is 1). We suggest this kernel since it ensures that $|\hat\del(z)-z|<\sig^{-1}$, see the Appendix.   However, we do  adopt the recommendation of Brown and Greenshtein (2009)
for a very slowly converging sequence $\sig_n=1/\log(n)$.

Denote
$     \bdf{\hat{\mu}}^I = \bdf{\tilde{\mu}}+ {\bdf{\hat{\delta}}}
    $.

\begin{theorem}\label{th:1} Under Assumptions \ref{ass1} and \ref{ass3}:
\begin{itemize}
\item[i)] \begin{equation*}
E||{\bdf{\hat{\mu}}}^I-\bdf{\mu}||^2
    \le E||{\bdf{{\mu}}}^I-\bdf{\mu}||^2+ o(n)\le E||{\bdf{{\hat\mu}}}-\bdf{\mu}||^2+ o(n).
\end{equation*}

\item[ii)] Under Assumptions \ref{ass1}--\ref{ass3}, if $\al_n\cip\al\in(0,1)$ and  $F_\nu^n$  converges weakly to a  distribution different from  $N(0,(1-\al)/\al)$, then there is  $c<1$ such that for  large enough $n$:
\begin{equation*}
   E||{\hat{\bdf{{\mu}}}}^I-\bdf{\mu}||^2\le cE||\hat{\bdf{\mu}}-\bdf{\mu}||^2.
\end{equation*}
\end{itemize}\end{theorem}

\begin{proof}\mbox{}

\begin{itemize}

\item[i)] The proof is given in the appendix

  \item[ii)] The same arguments used   to prove  part i)  may be used to prove a modification of Proposition \ref{prop:lin},  in which
 $\mu^I_i$  is replaced by $\hat{\mu}_i^I$, when assuming in addition that $ i/n\in{(\zeta,1-\zeta)}$ for any $\zeta\in(0,1)$.
\end{itemize}
\end{proof}

Part i) of the above theorem assures us that asymptotically the improved estimator does as good as $\hat{\bdf{\mu}}$; part ii)  implies
that in general  the improved estimator does asymptotically strictly better.   Obviously the asymptotic improvement is not always strict since the Kalman filter is optimal under a Gaussian state-space
model.

\section {Sequential estimation\label{sec:seq}}

We consider now the case $\scf_i=\sig(Y_1,\ldots ,Y_{i-1})$, $i \leq n$. The definition of the different estimators is the same as in the previous section with the necessary adaption to  the different information set.
Our aim is to find a  sequential estimator, denoted $\hat{\bdf{\mu}}^{IS}$,
that satisfies
$ E|| \hat{\bdf{\mu}}^{IS} - \bdf{\mu}||^2+ o(n) < E|| \hat{\bdf{\mu}} - \bdf{\mu}||^2$.
By a sequential estimator $\hat{\bdf{\mu}}^{IS}=(\psi^1,\ldots ,\psi^n)$ we mean that
$\psi^i\in\scf_i$, $i=1,\ldots ,n$.
A natural approach, which indeed  works, is to let $\psi^i= \tilde{\mu}_i + \hat{\delta}^i$, where $\hat{\delta}^i$ is defined as in  \eqref{eqn:brown1}, but with $\hat f=\hat f_i$ restricted to the available data $Z_1,\ldots ,Z_{i-1}$, $i=1,\ldots ,n$. Let $\hat{\bdf{\delta}}^S=(\hat{\delta}^1,\ldots ,\hat{\delta}^n)$.

We define:
\begin{equation*}
\hat{\bdf{\mu}}^{IS}= \tilde{\bdf{\mu}}+ \hat{\bdf{\delta}}^S.
\end{equation*}
Our main result in this section:

\begin{theorem}
\label{th:2}
Theorem \ref{th:1} holds with $\bdf{\hat\mu}^{IS}$ and $\bdf{\mu}^{IS}$ replacing $\bdf{\hat\mu}^{I}$ and $\bdf{\mu}^{I}$, respectively.
%
%
\end{theorem}

In order to prove Theorem \ref{th:2} we adapt
Lemma 1 of Samuel (1965). Samuel's result is stated for a compound decision
problem, i.e., the parameters are fixed, and the observations are independent. The result compares the performance of the optimal estimators in the sequential and retrospective procedures.
It is not clear a priori whether retrospective estimation is easier or more difficult than the sequential. On the one hand, the retrospective procedure is using more information when dealing with the $i$th parameter. On the other hand, the sequential estimator can adapt better to non-stationarity in the parameter sequence. Samuel proved that the latter is more important. There is no paradox here, since the retrospective procedure is optimal only under the assumption of permutation invariance, and under permutation invariance, the weak inequality in Lemma \ref{lem:samuel} below is, in fact, equality.

Our approach is to
 rephrase and generalize Samuel's lemma.
Let $\eta_1,\dots,\eta_n$ be $N(0,1)$ \iid random variables independent of $(\mu_i,\eps_i)$, $i=1,\dots,n$.  Let $L(\nu_i,\hat{\nu}_i)$ be the loss
for estimating $\nu_i$ by $\hat{\nu}_i$.
For every $i \leq n$ let ${\delta}^i$ be the decision function
that satisfies:
 \begin{equation*}
  \begin{split}
     \delta^i&= \argmin_\delta \E_{\bdf\eta} \sum_{j=1}^i L\bigl(\nu_j, \delta(\nu_i+\eta_i)\bigr)
     \\&=
\argmin_\delta \sum_{j=1}^i \E_{\bdf\eta} L\bigl(\nu_j,\delta(\nu_i+\eta_i)\bigr)
\\
&\equiv \argmin_\delta \sum_{j=1}^i R(\delta,\nu_j),\quad\text{say,}
   \end{split} \end{equation*}
where $\E_{\bdf\eta}$ is the expectation over $\bdf\eta$.  That is, $\delta^i$ is the functional that minimizes the sum of risks for estimating the components $\nu_1,\dots,\nu_i$, but it is applied only for estimating $\nu_i$.  The quantity $R(\delta^j,\nu_j)$ is the analog of $R(\phi_{F_j}, \theta_j)$ in Samuel's formulation. In analogy to Samuel (1965) we define
$ R_n\equiv n^{-1}\summ j1n R(\delta^n,\nu_j)  $, the empirical Bayes risk of the non-sequential problem.

\begin{lemma}
\label{lem:samuel}
$$ n^{-1} \summ j1n R(\delta^j,\nu_j) \leq R_n. $$
\end{lemma}

The proof of the lemma is formally similar to the proof of Lemma 1 of Samuel (1965).

\begin{proof}[Proof of Theorem \ref{th:2} ] From Theorem \ref{th:1}, $E\bigl(\hat{\delta}^i(Z) - \delta^i(Z)\bigr)^2 \rightarrow 0$. The last fact coupled with Lemma \ref{lem:samuel} implies part i) of Theorem \ref{th:2}. Part ii) is shown similarly to part ii)
of Theorem \ref{th:1}.
\end{proof}

\section{Simulations.}
\label{sec:sim}

We present now simulation results for the following state-space
model.

\begin{equation}
\begin{split}
\label{eqn:sim}
Y_{i} &=\mu _{i}+\varepsilon _{i}\text{ } \\
\mu _{i} &=\phi \mu _{i-1}+U_{i}, \quad i=1,\dots,n,
\end{split}
\end{equation}%
where $\varepsilon _{i}\sim N(0,1)$, $i=1,\ldots ,n$, are independent of each
other and of $U_{i}$, $i=1,\ldots ,n$.  The variables $U_i$, $i=1,\ldots ,n$ are independent,
$U_i= X_i I_i$ where $X_i \sim N(0,v)$ are independent, while $I_1,\dots,I_n$ are \iid Bernoulli with mean 0.1,  independent of each other and of $X_i$, $i=1,\ldots ,n$. We  study the twelve cases that
are determined by $\phi=0.25,0.75$ and $v=0,1,\ldots ,5$. In each case we investigate both the sequential
and the retrospective setups.

If   $U_1, \dots, U_n$,  were i.i.d Normal, the data would follow a Gaussian state-space, and the corresponding  Kalman filter estimator would be optimal. Since the $U_i$'s are not normal,  the corresponding AR(1) Kalman filter is not optimal (except in the degenerate case, $v=0$),  though it is optimal
among linear filters. This is reflected in our
simulation results where for the cases $v=0,1$ our ``improved'' method $\hat{\bdf{\mu}}^I$
performs slightly worse than $\hat{\bdf\mu}=\hat{\bdf{\mu}}^K$. It improves in all the rest. The above is stated and proved formally in the following
proposition.  It could also be shown indirectly by applying part ii) of  Theorem \ref{th:1}.

\begin{proposition}
\label{prop:strict}
 Consider the state-space model, as defined by  \eqref{eqn:sim}.
If $U_i$ are not normally distributed then
\begin{equation*}
   E||{\hat{\bdf{{\mu}}}}^I-\bdf{\mu}||^2\le  c E||\hat{\bdf{\mu}}^K-\bdf{\mu}||^2 ,
\end{equation*}
for a constant $c \in (0,1)$ and large enough $n$.
\end{proposition}

\begin{proof}

Given the estimators $\tilde{\mu}_i^K$ and $\hat{\mu}_i^K$, $i=1,\ldots ,n$, let $Z_i=Y_i-\tilde{\mu}_i$.
Then $Z_i= \mu_{i-1}+ U_i - \tilde{\mu}_i^K+\varepsilon_i = \nu_i + \varepsilon_i$. The distribution
$G^i$  of $\nu_i$ may be normal only if $U_i$ is normal, since $U_i$ is independent of $\mu_{i-1}$ and $\tilde{\mu}_i$.
The distributions  $G^i$ converge to a distribution $G$ as $i$ and $n-i$ approach infinity.
As before, $G$ is normal only if $U_i$ are normal. Now, asymptotically  optimal estimator for $\mu_i$ under squared loss
and given the observation $Y_i$, is
$\tilde{\mu}_i^K +\hat{\nu}_i$, where  $\hat{\nu}_i $ is the Bayes estimator under a prior  $G$ on $\nu_i$ and an observation
$Y_i \sim N(\nu_i,1)$. This Bayes estimator is linear and coincide with the KF estimator $\hat{\mu}^K_i$, only if $G$ is normal.

\end{proof}

Analogous discussion and situation are valid also in the sequential case.
 In our simulations  the parameters $\phi$ and $VAR\left(
U_{i}\right) $  are treated as known. Alternatively, maximum likelihood
estimation assuming (wrongly) normal inovations yields results similar to those reported in Table \ref{tab1}.

The simulation results in  Table \ref{tab1}  are for the case $n=500$. Each entry
is based on  100 simulations. In each realization we recorded
$||\hat{\bdf{\mu}}- \bdf{\mu}||^2$ and $||\hat{\bdf{\mu}}^I- \bdf{\mu}||^2$, and each
entry is based on the corresponding average. In order to speed the asymptotics we allowed a `warm up' of 100 observations prior to the $n=500$ in the sequential case,
we also allowed a `warm up' of 50 in both sides of the $n=500$ observations in the retrospective case.

\begin{table}
\caption{Mean Squred error of the two estimator for an autoregressive process with aperiodic
normal shocks}\label{tab1}
\begin{tabular}{|l|rrrrrr|rrrrrr|}
 \hline
$\phi $ & \multicolumn{6}{|c|}{\textsc{0.25}} & \multicolumn{6}{|c|}{\textsc{%
0.75}} \\ \hline
$v$ & \textsc{0} & \textsc{1} & \textsc{2} & \textsc{3} & \textsc{4} & \textsc{%
5} & \textsc{0} & \textsc{1} & \textsc{2} & \textsc{3} & \textsc{4} &
\textsc{5}
\\
&&&&&\multicolumn{7}{l}{\textit{Retrospective filter:}}&
\\ \hline
$\widehat{\mu }$$^\dag$ & {\small 0} & {\small 71} &
{\small 156} & {\small 226} & {\small 290} & {\small 333} & {\small 0} &
{\small 49} & {\small 147} & {\small 235} & {\small 301} & {\small 350}
\\
$\mu ^{I}$$^\ddag$ & {\small 23} & {\small 66}
& {\small 125} & {\small 148} & {\small 160} & {\small 177} & {\small 24} &
{\small 91} & {\small 166} & {\small 215} & {\small 253} & {\small 271}%
\\
&&&&&\multicolumn{7}{l}{\textit{Sequential filter:}}&  \\ \hline
$\widehat{\mu }$ $^\dag$ & {\small 0} & {\small 47} &
{\small 145} & {\small 234} & {\small 309} & {\small 355} & {\small 0} &
{\small 83} & {\small 187} & {\small 264} & {\small 325} & {\small 372} \\
$\mu ^{I}$ $^\ddag$ & {\small 39} & {\small 81}
& {\small 129} & {\small 147} & {\small 159} & {\small 158} & {\small 34} &
{\small 112} & {\small 184} & {\small 216} & {\small 239} & {\small 253}%
\\\hline
\multicolumn{13}{l}{\dag Kalman filter, \ddag Improved.}
\end{tabular}%
\end{table}

It may be seen that when the best linear filter is optimal or nearly optimal
(when $v=0$ or approximately so), our improved method is slightly worse than the Kalman filter estimator, however as $v$ increases, the advantage of the improved
method may become significant.

It seems that in the case $\phi=0.25$ the future observations are not very helpful,
and in our simulations there are cases
where the simulated risk of $\hat{\bdf{\mu}}$ in the sequential case is even smaller than the corresponding simulated risk of the retrospective case. This could be an artifact of the simulations, but also a result of noisy estimation of the coefficients $\beta_{ij}$ of the mildly informative future observations.

\section{ Real Data Example}\label{sec:real}
\begin{table}[t]
\caption{The retrospective case: Cross-validation estimation of the average squared risk.}
\label{tabA1}
\begin{tabular}{r|ddd}
\hline
$p=0.95$ & $AR(1)$ & $AR(2)$ \textit{\ } & $ARIMA(1,1,0)$ \\ \hline
{\small Kalman filter - }$\hat{\lambda}_{i}^{K}$ & 27.1 & 19.4 & 20.4 \\
{\small Improved method - }$\hat{\lambda}_{i}^{I}${\small \ } & 18.7 & 18.5
& 17.4 \\
{\small Naive method - }$\hat{\lambda}_{i}^{N}$ & 26.4 & 26.4 & 26.4%
\end{tabular}%
\end{table}

In this section we demonstrate the performance of our method on real data
taken from the FX  (foreign exchange) market in Israel. The data consists of the daily number of
swaps (purchase of of one currency for another with a given value date, done simultaneously with the selling back of the same amount with a different value date.  This way there is no foreign exchange risk)
in the OTC (over-the-counter)  Shekel/Dollar market.
We consider only the buys of magnitude 5 to 20 million dollars.
The time period is
January 2nd, 2009 to December 31st, 2013, a total of $n=989$ business days. The number of buys in each day is
24 on the average, with the range of 2--71.
In our analysis we used the first 100 observations as a `warm up',
similarly to the way it was done in our simulations section.

We denote by $X_{i}$, $i=1,\ldots,n$, the
number of buys on day $i$ and assume that $X_{i}\sim Po\left( \lambda
_{i}\right) .$ We transform the data by $Y_{i}=2 \sqrt{X_{i}+0.25}$ as
in Brown et al. (2010) and  Brown et al. (2013) in order to get an (approximately)
normal variable with variance $\sigma ^{2}=1.$

The
assumed model in this section is the following state space system of
equations:

 \eqsplit{
Y_{i} &=\mu _{i}+\varepsilon _{i} \\
\mu _{i} &\sim ARIMA\left( p,d,q\right),  \quad i=1,\ldots ,n,
  }
where $\mu _{i}=2 \sqrt{\lambda _{i}}$ and $\varepsilon _{i}\sim N(0,1)$ are independent of each other and of the $ARIMA\left(
p,d,q\right) $ process. We consider the following three special cases of $ARIMA\left(
p,d,q\right)$: $AR\left( 1\right) $, $AR\left( 2\right) $, and $%
ARIMA\left( 1,1,0\right) .$

Under each model there are induced Kalman filter estimators
$\tilde{\boldsymbol{\mu}}^K$, and $\hat{\boldsymbol{\mu}}^K$ that correspond to one step prediction and to the
update. Similarly, the improved estimator $\hat{\boldsymbol{\mu}}^I$ is defined. We denote the
sequential and retrospective estimators similarly with no danger  of confusion.

After estimating  $\mu_i$, we transform the result back
to get the estimator $\hat{\lambda}^J_i$ for $\lambda_i$, $\hat{%
\lambda}_{i}^{j}=0.25 \left( \hat{\mu}_{i}^{j}\right)^{2}$, $i=1,\ldots ,n$, $J\in\{\text{`}I\text{'},`K\text{'}\}$ where, $\hat{\mu}_{i}^{J}$ is the estimator of $\mu _{i}$ by method $%
J.$ We evaluate the performances of both estimation methods by the following
non-standard cross-validation method as described in Brown et al. (2013).  It is briefly explained in the following.

Let $p\in \left( 0,1\right) ,$ $p\thickapprox 1,$ and let $U_{1},\ldots ,U_{n}$
be independent given $X_{1},\ldots ,X_{n}$, where $U_{i}\sim B\bigr(
X_{i},p\bigr) $ are Binomial variables. It is known that $%
U_{i}\sim Po\bigl( p\lambda _{i}\bigr) $ and $V_{i}=X_{i}-U_{i}\sim
Po\bigl( \left( 1-p\right) \lambda _{i}\bigr) $ and they are independent
given $\lambda _{1},\ldots ,\lambda _{n}.$ We will use the `main' sub-sample $%
U_{1},\ldots ,U_{n}$ for the construction of both estimators (Kalman filter and
Improvement) while the `auxiliary' sub-sample $V_{1},\ldots ,V_{n}$ is used for
validation. Consider the following loss function,
\begin{equation*}
\begin{split}
\rho \left( J;\mathbf{U,V}\right)
&=\frac{1}{n}\overset{n}{\underset{i=1}{%
\sum }}\Bigl( \frac{\hat{\lambda}_{i}^{J}}{p}-\frac{V_{i}}{\left( 1-p\right)
}\Bigr) ^{2}
\\
&= \frac{1}{n p^{2}}\overset{n}{\underset{i=1}{\sum }}\left(
\hat{\lambda}_{i}^{J}-p\lambda _{i}\right) ^{2}
+\frac{1}{n (1-p)^{2}}\overset{n}{\underset{i=1}{\sum }}\left(
V_i-(1-p)\lambda _{i}\right) ^{2}
\\
&\hspace{3em}-\frac2{n} \summ i1n \Bigl(\frac{\hat{\lambda}_{i}^{J}}{p}-\lambda _{i}\Bigr)\Bigl(
\frac{V_i}{(1-p)}-\lambda _{i}\Bigr)
\\
&=\frac{1}{n p^{2}}\overset{n}{\underset{i=1}{\sum }}\left(
\hat{\lambda}_{i}^{J}-p\lambda _{i}\right) ^{2} +A_{n}+R_{n}\left( J\right), \quad J\in\{'K','I'\}.
\end{split}\end{equation*}%
The term $R_{n}\left( J\right)=\O_p(n^{-1/2}) $ and will be ignored.  We  estimate $A_n$ by the method of moments:
\begin{equation*}
  \begin{split}
   \hat A_n&=\frac{1}{n(1-p)^2}\summ i1n V_i.
   \end{split} \end{equation*}

 We
repeat the cross-validation process 500 times and average  the computed values of  $\rho \left( J;\mathbf{U,V}\right)-\hat A_{n}$.
When $p$ is close to 1, the average obtained is a plausible  approximation of the average squared
risk in estimating $\lambda_i$, $i=101,\ldots ,989$.
By the above method we approximated also the
average risk of the naive estimator $\hat{\lambda}_{i}^{N}=X_i$, $i=101,\ldots ,989$.
The approximations
for the retrospective and sequential cases, are displayed in Tables
\ref{tabA1} and \ref{tabA3}. The estimated ARIMA coefficient for the various models are given in Table \ref{tabA2}.

\begin{table}[t]
\caption{The retrospective case: Parameter estimation }
\label{tabA2}
\begin{tabular}{l|ddd}
\hline
$p=0.95$ & $AR(1)$ & $AR(2)$ & $ARIMA(1,1,0)$\textit{\ } \\
\hline
$\alpha $ & 12.38 & 14.218 & 0.01 \\
$\phi _{1}$ & -0.28 & -0.341 & -0.6 \\
$\phi _{2}$ &  & -0.124 &\\
$\sig^2$ & 3.4 & 3.4 &4.7\\
\end{tabular}%
\end{table}

\begin{table}[t]
\caption{The sequential case: Cross-validation approximation of the average squared risk}
\label{tabA3}

\begin{tabular}{l|ddd}
 \hline
$p=0.95$ & $AR(1)$ & $AR(2)$ & $ARIMA(1,1,0)$ \\ \hline
{\small Kalman filter - }$\hat{\lambda}_{i}^{K}$ & 19.2 & 19.2 & 21.2 \\
{\small Improved method - }$\hat{\lambda}_{i}^{I}${\small \ } & 19.0 & 19.2 &
22.6 \\
{\small Naive method - }$\hat{\lambda}_{i}^{N}$ & 26.4 & 26.4 & 26.4%
\end{tabular}%
\end{table}

From Table \ref{tabA1} we can observe that in the retrospective case the
improved method does uniformly better than the naive estimator and
the Kalman filter. In fact, in all except a small deterioration  under the ARIMA(1,1,0) with sequential filtering, the performance of the improved method is quite uniform, showing its robustness against model miss-specification.

Somewhat surprising is  that the Kalman filter under AR(1) with retrospective estimation does not do better than the naive filter, but do so considerably in  the sequential case. The reason is that the AR(1) model does not fit the data. When it is enforced on the data, the Kalman filter gives too much weight to the surrounding data, and too little to the ``model free'' naive estimator. This result show the robustness of our estimator.

In fact, we did a small simulation, where the process was AR(2), with the parameters as estimated for the data. When an AR(1) was fitted to the data, the retrospective Kalman filter was strictly inferior to the sequential one.

In the sequential case, Table \ref{tabA3},  the
improved method does better than the naive method, but contrary to the
non-sequential case, it improves upon the Kalman filter only in the AR(1) and
AR(2) models, while in the ARIMA(1,1,0) model the Kalman filter does slightly
better.

\section{Appendix: Proof of Theorem \ref{th:1}}

Note that by \eqref{eqn:imp},  obviously, $E||{\bdf{{\mu}}}^I-\bdf{\mu}||^2 < E||\hat{\bdf{\mu}}-\bdf{\mu}||^2+ o(n)$. Thus,  in order to obtain $E||{\bdf{\hat{\mu}}}^I-\bdf{\mu}||^2 < E||\hat{\bdf{\mu}}-\bdf{\mu}||^2+ o(n)$
it is enough to show that $E|| \bdf{\hat{\delta}}-\bdf{\delta}||^2 =o(n).$

First recall:
 \eqsplit{
    K(x) &= \frac{2}{\bigl(e^{x}+e^{-x}\bigr)^2}
    \\
    K'(x) &= -4\frac{e^{x}-e^{-x}}{\bigl(e^{x}+e^{-x}\bigr)^3}
    \\
    \frac{K'(x)}{K(x)} &= -2 \frac{e^{x}-e^{-x}}{e^{x}+e^{-x}}\in(-2,2).
  }
Thus
\begin{equation}\label{KpoKb}
\sup_z\Bigl|\frac{\hat f_Z'(z)}{\hat f_Z(z)}\Bigr|<2\sig_n^{-1}.
\end{equation}

By Assumption \ref{ass1}, if we replace $\ti\mu_i$ by a similar (unobserved) function $\ti\mu_i^*$, where $Y_j$, $j\in \scd_i$, $|j-i|>n^{\gamma}$ is replaced by $\mu_j$, then $\max_i|\ti\mu_i^*-\ti\mu_i|=\O_p\bigl( n^{-\gamma}\bigr) \max_i|\eps_i|=\o_p(n^{-\gamma/2})$, for any $\gamma>0$. Define  $\nu_i^*$ and $Z_i^*$ as in \eqref{nuZ}, but where $\ti\mu_i$ is replaced by $\ti\mu_i^*$, $i=1,\dots,n$. Since $|\hat f''|\leq \sig_n^{-3}$, $\max_i|Z_i-Z_i^*|$ is ignorable for our approximations. In the following all variables are replaced by their $*$ version, but we drop the $*$ for simplicity.

Let $L_n= \log n$. By Assumption \ref{ass3} with probability greater than $1-L_n^{-4}$, all but $n/L_n^4$ of the $Z_i$s are in $(-L_n,L_n)$, hence, their density is mostly not too small:
\begin{equation}\label{fzNs}
\int \ind\bigl(f_Z(z)<L_n^{-3}\bigr) f_Z(z)dz<L_n^{-2}.
\end{equation}

Let
 \eqsplit{
 \bar\varphi(z)&=\varphi*K_{\sig_n}=\int K\Bigl({\frac{z-e}{\sig_n}}\Bigr)\varphi(e)de
 \\
  \bar f_Z&=f_Z*K_{\sig_n}
  =  \frac 1n \summ j1n \bar\varphi(z-\nu_j)
  }
Then
 \eqsplit{
    \hat f_Z(z)-\bar f_Z(z)
    &=  \frac{1}{\sig_n n} \summ j1n K\Bigl(\frac{z-Z_j}{\sig_n}\Bigr)
    - \frac1{n} \summ j1n \bar\varphi(z-\nu_j)
    \\
    &=  \frac{1}{\sig_n n} \summ j1n \Bigl\{K\Bigl(\frac{z-\nu_j-\eps_j}{\sig_n}\Bigr) - \int K\Bigl(\frac{z-\nu_j-e}{\sig_n}\Bigr)\varphi(e)de\Bigr\}.
      }
Since $\eps_j$ and $\nu_j$ are independent, this difference has mean 0. Moreover, since $(\nu_j^*,\eps_j)$ and $(\nu_k^*,\eps_k)$ are independent for $|j-k|>M$, the above expression has variance of order $M(n\sig_n)^{-1}$. The difference $\|\bar f_Z-f_Z\|_\en$ is of order $\sig_n^2$.

A similar expansion works for $\hat f_Z'$, except that the variance now is of order $M(n\sig_n^3)^{-1}$. Regular large deviation argument shows that when $f_Z(Z_i)>L_n^{-3}$ then $\hat f_Z(Z_i)<L_n^{-3}/2$ with exponential small probability. By \eqref{KpoKb} and \eqref{fzNs} it follows that the the approximation $\hat f_Z'(Z_i)/\hat f_Z(Z_i)=f_Z'(Z_i)/ f_Z(Z_i) + \o_p(1)$ holds not only in the mean but also in the mean square, since the exception probability are smaller than $\sig_n^2$.

\vspace{3ex}
{\bf \Large References}

\begin{list}{}{\setlength{\itemindent}{-1em}\setlength{\itemsep}{0.5em}}
\item
Brockwell, P.J. and Davis, R.A. (1991). Time Series: Theory and Methods.
Second edition. Springer.f
\item
Brown, L. D. (1971). Admissible estimators, recurrent diffusions, and insoluble boundary
value problems. {\it Ann.Math.Stat.} {\bf 42}, 855-904.
\item
Brown, L.D., Cai, T., Zhang, R., Zhao, L., Zhou, H. (2010). The
root-unroot algorithm for density
estimation as implemented via wavelet block thresholding. {\it Probability and Related Fields}, {\bf 146}, 401-433.
\item
Brown, L.D. and Greenshtein, E. (2009). Non parametric
empirical Bayes and compound decision
approaches to estimation of high dimensional vector of normal
means. {\it Ann. Stat.} {\bf 37}, No 4, 1685-1704.
\item
Brown L.D., Greenshtein, E. and Ritov, Y. (2013). The Poisson compound decision revisited. {\it JASA.} {\bf 108} 741-749.
\item
Cohen, N., Greenshtein E., and Ritov, Y. (2012). Empirical Bayes in the presence of explanatory
variables.  {\it Statistica Sinica.} {\bf 23}, No. 1, 333-357.
\item
Copas, J.B. (1969). Compound decisions  and empirical Bayes (with discussion). {\it JRSSB} {\bf 31} 397-425.
\item
Fay, R.E. and Herriot, R. (1979). Estimates of income for small places: An application of James-Stein
procedure to census data. {\it JASA}, {\bf 74}, No. 366, 269-277.
\item
Greenshtein, E. and Ritov, Y. (2008). Asymptotic efficiency of
simple
decisions
for the compound decision problem. {The 3'rd
Lehmann Symposium. IMS Lecture
Notes Monograph Series}, J.Rojo, editor. 266-275.
\item
Jiang, W. and Zhang, C.-H. (2010). Empirical Bayes in-season prediction of baseball batting average. {\it Borrowing Strength: Theory Powering Application-A festschrift for L.D. Brown}
J.O. Berger, T.T. Cai, I.M. Johnstone, eds. IMS collections  {\bf 6}, 263-273.
\item
Koenker, R. and Mizera, I. (2012). Shape constraints , compound decisions and empirical Bayes rules. Manuscript.
\item
Robbins, H. (1951). Asymptotically subminimax solutions  of
compound decision problems. {\it Proc. Second Berkeley Symp.}
131-148.
\item
Robbins, H. (1955). An Empirical Bayes approach to statistics.
{\it  Proc. Third Berkeley Symp.} 157-164.
\item
Robbins, H. (1964). The empirical Bayes approach to statistical
decision problems. {\it Ann.Math.Stat.}
{\bf 35}, 1-20.
\item
Samuel, E. (1965). Sequential Compound Estimators. {\it Ann.Math. Stat.}
{\bf 36}, No 3, 879-889.
\item
Zhang, C.-H.(2003). Compound decision theory and empirical
Bayes methods.(invited paper). {\it Ann. Stat.} {\bf 31}
379-390.

\end{list}

\end{document}